\documentclass[12pt,reqno]{amsart}

\usepackage{amssymb}

\newtheorem{theorem}{Theorem}[section]
\newtheorem{proposition}[theorem]{Proposition}
\newtheorem{lemma}[theorem]{Lemma}
\newtheorem{corollary}[theorem]{Corollary}

\theoremstyle{remark}
\newtheorem{remark}[theorem]{Remark}

\numberwithin{equation}{section}

\begin{document}

\title[Projective structures and differential operators]{Projective structures 
on Riemann surface and natural differential operators}

\author[I. Biswas]{Indranil Biswas}

\address{School of Mathematics, Tata Institute of Fundamental
Research, Homi Bhabha Road, Mumbai 400005, India}

\email{indranil@math.tifr.res.in}

\author[S. Dumitrescu]{Sorin Dumitrescu}

\address{Universit\'e C\^ote d'Azur, CNRS, LJAD, France}

\email{dumitres@unice.fr}



\date{}

\begin{abstract}
We investigate the holomorphic differential operators on a Riemann surface $M$. This is done by
endowing $M$ with a projective structure. Let $\mathcal L$ be a theta characteristic on $M$.
We explicitly describe the jet bundle $J^k(E\otimes {\mathcal L}^{\otimes n})$, where $E$ is
a holomorphic vector bundle over $M$ equipped with a holomorphic connection, for
all $k$ and $n$. This provides a description of global holomorphic differential operators
from $E\otimes {\mathcal L}^{\otimes n}$
to another holomorphic vector bundle $F$ using the natural isomorphism
$\text{Diff}^k(E\otimes {\mathcal L}^{\otimes n},\, F)\,=\, F\otimes
(J^k(E\otimes {\mathcal L}^{\otimes n}))^*$.
\end{abstract}

\maketitle

\section{Introduction}\label{sec1}

The study of natural differential operators between natural vector bundles is a major topic in global 
analysis on manifolds (see, for instance, \cite{Ni}, \cite{Te} and \cite{KMS}).

Our aim here is to study holomorphic differential operators between natural holomorphic 
vector bundles on a Riemann surface $M$. To be able to formulate their description, we make use 
of the existence of compatible projective structures on the Riemann surface $M$ (see \cite{Gu}
for projective structure on Riemann surfaces).

We recall that a projective structure on $M$ is given by a holomorphic coordinate atlas such 
that all the transition functions are restrictions of M\"obius transformations of
the projective line ${\mathbb C}{\mathbb P}^1$. The above mentioned approach 
to holomorphic differential operators on $M$ entails a detailed understanding of the action of 
$\text{SL}(2,{\mathbb C})$ on the jet bundles on ${\mathbb C}{\mathbb P}^1$.

Let $\mathbb V$ be a complex vector space of dimension two. For notational convenience, denote by $X$ the space 
${\mathbb P}({\mathbb V})$ of complex lines in $\mathbb V$. Let $L$ be the tautological
holomorphic line bundle of degree one on $X$. The rank two
trivial holomorphic vector bundle $X\times {\mathbb V}$ on $X$ with fiber $\mathbb V$ will be 
denoted by ${\mathbb V}_1$, while the $j$--fold symmetric product $\text{Sym}^j({\mathbb V}_1)$ of
${\mathbb V}_1$ will be 
denoted by ${\mathbb V}_j$. The holomorphically trivial line bundle $X\times {\mathbb C}$ on $X$ 
will be denoted by ${\mathbb V}_0$.

We prove the following (see Theorem \ref{thm1}):

\begin{theorem}\label{thm0}
If $n\, <\, 0$ or $n\, \geq\, k$, then
$$
J^k(L^n)\,=\, L^{n-k}\otimes {\mathbb V}_k\, .
$$
If $k\, > \, n\, \geq\, 0$, then
$$
J^k(L^n)\,=\, {\mathbb V}_n\oplus (L^{-(k+1)}\otimes {\mathbb V}_{k-n-1})\, .
$$
Both the isomorphisms are ${\rm SL}({\mathbb V})$--equivariant.
\end{theorem}

The natural forgetful homomorphism $J^k(L^n)\, \longrightarrow\, J^{k-1}(L^n)$ of
holomorphic vector bundles equipped with action of ${\rm SL}({\mathbb V})$ is described in
Proposition \ref{prop2}, Proposition \ref{prop3} and Proposition \ref{prop4}.

Since all the above mentioned isomorphisms and homomorphisms are ${\rm SL}({\mathbb 
V})$--equivariant, we are able to translate them into the set-up of a Riemann surface equipped 
with a projective structure. We recall that any projective structure on a Riemann surface $M$ has
an underlying holomorphic atlas on $M$ for which the transition maps are restrictions of elements in ${\rm 
SL}({\mathbb V})$ (see Section \ref{s4}).

Let $M$ be a compact Riemann surface equipped with a projective structure $P$. Using $P$,
the above mentioned tautological line bundle $L$ on $X$ produces a holomorphic line bundle $\mathbf L$
on $M$ with the property that ${\mathbf L}^{\otimes 2}\,=\, TM$. Moreover, using $P$, the trivial vector
bundle ${\mathbb V}_1$ produces a holomorphic vector bundle ${\mathcal V}_1$ of rank two over $M$
that fits in the following short exact sequence of holomorphic vector bundles on $M$:
$$
0\, \longrightarrow\, {\mathbf L}^*\, \longrightarrow\, {\mathcal V}_1\,\longrightarrow\,
{\mathbf L} \, \longrightarrow\, 0\, .
$$
This involves fixing a nonzero element of $\bigwedge^2\mathbb V$. The vector bundle ${\mathcal V}_1$
also has a holomorphic connection given by $P$.

The symmetric power $\text{Sym}^j({\mathcal V}_1)$ will be denoted by ${\mathcal V}_j$,
while the trivial line bundle ${\mathcal O}_M$ will be denoted by ${\mathcal V}_0$.

Let us now describe our results about natural differential holomorphic operators on a Riemann surface endowed with a projective structure.

The following is deduced using Theorem \ref{thm0} (see Corollary \ref{cor1}):

\begin{corollary}\label{cor0}
Let $M$ be a compact connected Riemann surface endowed with a projective structure $P$.
If $n\, <\, 0$ or $n\, \geq\, k$, then there is a canonical isomorphism
$$
J^k({\mathbf L}^n)\,=\, {\mathbf L}^{n-k}\otimes {\mathcal V}_k\, ,
$$
where ${\mathcal V}_k$ is defined above.

If $k\, > \, n\, \geq\, 0$, then there is a canonical isomorphism
$$
J^k({\mathbf L}^n)\,=\, {\mathcal V}_n\oplus ({\mathbf L}^{-(k+1)}\otimes {\mathcal V}_{k-n-1})\, .
$$
\end{corollary}

This enables us to deduce the following result (see Lemma \ref{lem2}):

\begin{lemma}\label{lem0}
Let $M$ be a compact connected Riemann surface endowed with a projective structure $P$. 
Consider $E$ a holomorphic vector bundle on $M$ equipped with a holomorphic connection.

If $n\, <\, 0$ or $n\, \geq\, k$, then there is a canonical isomorphism
$$
J^k(E\otimes {\mathbf L}^n)\,=\, E\otimes{\mathbf L}^{n-k}\otimes {\mathcal V}_k\, .
$$
If $k\, > \, n\, \geq\, 0$, then there is a canonical isomorphism
$$
J^k(E\otimes{\mathbf L}^n)\,=\, E\otimes ({\mathcal V}_n\oplus ({\mathbf L}^{-(k+1)}\otimes
{\mathcal V}_{k-n-1}))\, .
$$
\end{lemma}

Let $E$ be a holomorphic vector bundle on $M$ equipped with a holomorphic connection, as in
Lemma \ref{lem0}. Given a holomorphic vector bundle $F$ on $M$, denote by ${\rm 
Diff}^k_M(E\otimes {\mathbf L}^n,\, F)$ the sheaf of holomorphic differential operators of 
order $k$ from $E\otimes {\mathbf L}^n$ to $F$.

The associated symbol homomorphism
$$
{\rm Diff}^k_M(E\otimes {\mathbf L}^n,\, F)\, \longrightarrow\,
{\rm Hom}(E\otimes {\mathbf L}^n,\, F)\otimes (TM)^{\otimes k}
$$
is described in Corollary \ref{cor3} and Corollary \ref{cor6}.

Fix integers $k$, $n$ and $l$ such that at least one of the following two conditions is valid:
\begin{enumerate}
\item $n \,<\, 0$, or

\item $n,\, l \,\geq\, k$.
\end{enumerate}
Also fix a section
\begin{equation}\label{t}
\theta_0\, \in\, H^0(M,\, \text{End}(E)\otimes{\mathbf L}^l)\, .
\end{equation}
Using $\theta_0$, and a projective structure $P$ on $M$, we construct a holomorphic
differential operator
$$
{\mathcal T}_\theta \, \in\, H^0(M,\, \text{Diff}^k(E\otimes{\mathbf L}^n,\,
E\otimes{\mathbf L}^{l+n-2k}))\, .
$$
The following theorem shows that ${\mathcal T}_\theta$ is a lift of the symbol $\theta_0$
to a holomorphic differential operator (see Theorem \ref{thm2}):

\begin{theorem}\label{thm01}
The symbol of the above holomorphic differential operator ${\mathcal T}_\theta$
is the section $\theta_0$ in \eqref{t}.
\end{theorem}

\section{Preliminaries}\label{se2}

\subsection{Jet bundles and differential operators}

Let $X$ be a connected Riemann surface.
The holomorphic cotangent bundle of $X$ will be denoted by $K_X$.
Consider the Cartesian product $X\times X$. For $i\,=\, 1,\, 2$, let
$$
p_i\, :\, X\times X\, \longrightarrow\, X
$$
be the natural projection to the $i$--th factor. Let
\begin{equation}\label{e-2}
\Delta\, :=\, \{(x,\, x)\, \mid\, x\, \in\, X\}\, \subset\, X\times X
\end{equation}
be the diagonal divisor. At times we shall identify $\Delta$ with $X$
using the mapping $x\, \longmapsto\, (x,\, x)$, where $x\, \in\, X$. For any nonnegative integer
$k$, consider the natural short exact sequence of coherent analytic sheaves on $X\times X$
\begin{equation}\label{e0n}
0\, \longrightarrow\, {\mathcal O}_{X\times X} (-(k+1)\Delta)\, \longrightarrow\,
{\mathcal O}_{X\times X} (-k\Delta)
\end{equation}
$$
\longrightarrow\, {\mathcal O}_{X\times X} (-k\Delta)/{\mathcal O}_{X\times X} (-(k+1)\Delta)
\, \longrightarrow\, 0\, .
$$
Using the Poincar\'e adjunction formula, we have
\begin{equation}\label{e0}
{\mathcal O}_{X\times X} (-k\Delta)/{\mathcal O}_{X\times X} (-(k+1)\Delta)
\,=\, K^{\otimes k}_X\, ;
\end{equation}
here $X$ is identified with $\Delta$ as explained above. Let $z$ be a holomorphic coordinate
function on an open subset $U\, \subset\, X$. The isomorphism in \eqref{e0} over 
$(U\times U)\bigcap\Delta$ sends
the section $(z\circ p_2 -z\circ p_1)^k$ of ${\mathcal O}_{X\times X} (-k\Delta)/{\mathcal
O}_{X\times X} (-(k+1)\Delta)$ over $(U\times U)\bigcap\Delta$ to the section $(dz)^{\otimes k}$ of
$K^{\otimes k}_X\vert_U$ (after identifying $U$ with $(U\times U)\bigcap\Delta$ in the above fashion).

Let $W$ be a holomorphic vector bundle on $X$ of rank $r_W$.
For any nonnegative integer $k$, consider the coherent analytic subsheaf
$$
(p^*_2 W)\otimes {\mathcal O}_{X\times X} (-(k+1)\Delta)\, \subset\,
p^*_2 W\, .
$$
The $k$--th order jet bundle $J^k(W)$ is defined to be
the direct image
\begin{equation}\label{e1}
J^k(W)\, :=\, p_{1*}((p^*_2 W)/((p^*_2 W)\otimes {\mathcal O}_{X\times X}
(-(k+1)\Delta)))\, \longrightarrow\, X\, ,
\end{equation}
which is a holomorphic vector bundle on $X$ of rank $(k+1)r_W$.

For any $k\, \geq\, 1$, using the natural inclusions of coherent analytic sheaves on $X\times X$
$$
(p^*_2 W)\otimes {\mathcal O}_{X\times X} (-(k+1)\Delta)\, \subset\,
(p^*_2 W)\otimes {\mathcal O}_{X\times X} (-k\Delta)\, \subset\, p^*_2 W
$$
together with the isomorphism in \eqref{e0}, we obtain a short exact sequence of holomorphic vector
bundles on $X$
\begin{equation}\label{e2}
0\, \longrightarrow\, W\otimes K^{\otimes k}_X \, \stackrel{\iota}{\longrightarrow}\,J^k(W)
\, \stackrel{\phi_k}{\longrightarrow}\,J^{k-1}(W) \, \longrightarrow\, 0\, .
\end{equation}

Let $V$ be a holomorphic vector bundle on $X$. Let $\text{Diff}^k(W, V)$ denote
the sheaf of holomorphic differential operators of order $k$ from $W$ to $V$.
It is the sheaf of holomorphic sections of
\begin{equation}\label{dod}
\text{Diff}^k(W, V)\, :=\, \text{Hom}(J^k(W),\, V)\, =\, V\otimes J^k(W)^*\, .
\end{equation}
Consider the homomorphism
\begin{equation}\label{e2p}
\text{Diff}^k(W, V)\,=\, V\otimes J^k(W)^*\, \stackrel{\text{Id}_V\otimes \iota^*}{\longrightarrow}\,
V\otimes (W\otimes K^{\otimes k}_X)^*\,=\, (TX)^{\otimes k}\otimes \text{Hom}(W,\, V)\, ,
\end{equation}
where $\iota^*$ is the dual of the homomorphism $\iota$ in \eqref{e2}. This homomorphism
$\text{Id}_V\otimes \iota^*$ in \eqref{e2p} is known as the \textit{symbol map}on differential operators.

\subsection{Line bundles on the projective line}

Let $\mathbb V$ be a complex vector space of dimension two. Now set
$$
X\, :=\, {\mathbb P}({\mathbb V})
$$
to be the compact Riemann surface of genus zero parametrizing all the
complex lines in $\mathbb V$. Let
\begin{equation}\label{l}
L\, :=\, {\mathcal O}_X(1) \,=\, {\mathcal O}_{{\mathbb P}({\mathbb V})}(1)
\, \longrightarrow\, X
\end{equation}
be the tautological holomorphic line bundle of degree one whose fiber over the
point of ${\mathbb P}({\mathbb V})$ corresponding to a line $\xi\, \subset\, {\mathbb V}$ is the quotient
line ${\mathbb V}/\xi$. For any integer $n$, by $L^n$ we shall denote
the holomorphic line bundle
\begin{itemize}
\item $L^{\otimes n}$, if $n\, >\, 0$,

\item $(L^{\otimes -n})^*$, if $n\, <\, 0$, and

\item ${\mathcal O}_X$ (the trivial holomorphic line bundle), if $n\, =\, 0$.
\end{itemize}

A theorem of Grothendieck asserts that any holomorphic vector bundle $F$ on $X$ of
rank $r$ is of the form
$$
F\, =\, \bigoplus_{i=1}^r L^{a_i}\, ,
$$
where $a_i\, \in\, \mathbb Z$ \cite[p.~122, Th\'eor\`eme 1.1]{Gr}; moreover, the above 
integers $\{a_i\}$ are uniquely determined, up to a permutation, by $F$ \cite[p.~122, 
Th\'eor\`eme 1.1]{Gr}, \cite[p.~315, Theorem 2]{At}. Recall that the holomorphic cotangent
bundle $K_X$ is holomorphically isomorphic to $L^{-2}$.

Let
\begin{equation}\label{l2}
{\mathbb L}\, \longrightarrow\, X
\end{equation}
be the tautological holomorphic line bundle of degree $-1$ whose fiber over the
point of ${\mathbb P}({\mathbb V})$ corresponding to a line $\xi\, \subset\, {\mathbb V}$ is
$\xi$ itself. This $\mathbb L$ is holomorphically isomorphic to $L^{-1}$. We note that there is no
natural isomorphism between $L^{-1}$ and $\mathbb L$.

Let $\text{SL}({\mathbb V})$ be the complex Lie group of
dimension three consisting of all linear automorphisms of 
$\mathbb V$ that act trivially on the determinant line $\bigwedge^2 \mathbb V$. This group 
$\text{SL} ({\mathbb V})$ acts on ${\mathbb P}(\mathbb V)$ as follows: the action of any $A\, 
\in\, \text{SL}({\mathbb V})$ sends a line $\xi\, \subset\, \mathbb V$ to the line $A(\xi)$. 
The corresponding action homomorphism $$\text{SL}({\mathbb V})\, \longrightarrow\, 
\text{Aut}(X)$$ is holomorphic, and its kernel is $\pm \text{Id}_{\mathbb V}$. The action of 
$\text{SL}({\mathbb V})$ on $X$ has a tautological lift to an action of $\text{SL}({\mathbb V})$ on 
$K^{\otimes m}_X$ for every integer $m$.

The above action of $\text{SL}({\mathbb V})$ on ${\mathbb P}(\mathbb V)$ clearly lifts to
the line bundle $L$. Indeed, the action of any $A\, \in\,
\text{SL}({\mathbb V})$ on $\mathbb V$ produces an isomorphism
$$
{\mathbb V}/\xi\, \stackrel{\sim}{\longrightarrow}\, {\mathbb V}/A(\xi)
$$
for every line $\xi\,\subset\,\mathbb V$. The resulting action of $\text{SL}({\mathbb V})$ on $L$
produces an action of $\text{SL}({\mathbb V})$ on the tensor product $L^n$ for every $n$.

The diagonal action of $\text{SL}({\mathbb V})$ on $X\times X$ preserves the divisor $\Delta$ 
defined in \eqref{e-2}. This action of $\text{SL}({\mathbb V})$ on $X\times X$, and the action of 
$\text{SL}({\mathbb V})$ on $L^n$, together produce an action of $\text{SL}({\mathbb V})$ on the
jet bundle $J^k(L^n)$ defined in \eqref{e1}. The homomorphism $\phi_k$ in \eqref{e2} is 
in fact $\text{SL}({\mathbb V})$--equivariant.

\section{Description of jet bundles on the projective line}

Consider the short exact sequence of coherent analytic sheaves on $X\times X\,=\,
{\mathbb P}({\mathbb V})\times {\mathbb P}({\mathbb V})$
\begin{equation}\label{e3}
0\, \longrightarrow\, (p^*_2 L^n)\otimes{\mathcal O}_{X\times X} (-(k+1)\Delta)
\, \longrightarrow\, p^*_2 L^n \, \longrightarrow\, (p^*_2 L^n)\vert_{(k+1)\Delta}
\, \longrightarrow\, 0\, ,
\end{equation}
where $L$ is defined in \eqref{l}, and
$p_i$ as before is the projection $X\times X\, \longrightarrow\, X$ to the $i$--th factor.
{}From \eqref{e1} we know that the jet bundle $J^k(L^n)$ is the direct image of $(p^*_2 L^n)\vert_{(k+1)\Delta}$
to $X$ under the projection $p_1$. We shall investigate $J^k(L^n)$ using the long exact sequence
of direct images for the short exact sequence of sheaves in \eqref{e3}.

Let
$$
0\, \longrightarrow\, R^0p_{1*}((p^*_2 L^n)\otimes{\mathcal O}_{X\times X} (-(k+1)\Delta))
\, \longrightarrow\, R^0p_{1*} p^*_2 L^n \, \longrightarrow\, R^0p_{1*}((p^*_2 L^n)\vert_{(k+1)\Delta})
$$
\begin{equation}\label{e4}
=:\, J^k(L^n)\, \longrightarrow\, R^1p_{1*}((p^*_2 L^n)\otimes{\mathcal O}_{X\times X} (-(k+1)\Delta))
\, \longrightarrow\, R^1p_{1*} p^*_2 L^n \, \longrightarrow\, 0
\end{equation}
be the long exact sequence of direct images, under $p_1$, associated to the short exact sequence in 
\eqref{e3}. Note that $$R^1p_{1*}((p^*_2 L^n)\vert_{(k+1)\Delta})\,=\, 0\, ,$$ because the 
support of $(p^*_2 L^n)\vert_{(k+1)\Delta}$ is finite over $X$ (for the projection $p_1$).

Since the diagonal action of $\text{SL}({\mathbb V})$ on $X\times X$ preserves the divisor 
$\Delta$, each sheaf in \eqref{e4} is equipped with an action of $\text{SL}({\mathbb V})$ that 
lifts the action of $\text{SL}({\mathbb V})$ on $X$. Moreover, all the homomorphisms in 
the exact sequence \eqref{e4} are in fact $\text{SL}({\mathbb V})$--equivariant.

The direct image $R^ip_{1*} p^*_2 L^n$ in \eqref{e4} is evidently the trivial holomorphic vector bundle
\begin{equation}\label{ci}
R^ip_{1*} p^*_2 L^n\,=\, X\times H^i(X,\, L^n)
\end{equation}
over $X$ with fiber $H^i(X,\, L^n)$ for $i\,=\, 0,\, 1$.

Next, to understand $R^ip_{1*}((p^*_2 L^n)\otimes{\mathcal O}_{X\times X} (-(k+1)\Delta))$
in \eqref{e4}, we first
note that the holomorphic line bundle ${\mathcal O}_{X\times X} (\Delta)$ is holomorphically
isomorphic to $(p^*_1L)\otimes (p^*_2L)$. To construct an explicit isomorphism between
${\mathcal O}_{X\times X} (\Delta)$ and $(p^*_1L)\otimes (p^*_2L)$, fix a nonzero element
\begin{equation}\label{e5}
0\, \not=\, \omega\, \in\, \bigwedge\nolimits^2 {\mathbb V}\, .
\end{equation}
We have $H^0(X,\, L)\,=\, \mathbb V$. So
$$
H^0(X\times X,\, (p^*_1L)\otimes (p^*_2L))\,=\, H^0(X,\, L)\otimes H^0(X,\, L)\,=\,
{\mathbb V}\otimes {\mathbb V}\, ,
$$
and hence $\omega\, \in\, \bigwedge\nolimits^2 {\mathbb V}\, \subset\, {\mathbb V}\otimes {\mathbb V}$
in \eqref{e5} defines a holomorphic section
\begin{equation}\label{som}
\widehat{\omega}\, \in\, H^0(X\times X,\, (p^*_1L)\otimes (p^*_2L))\, .
\end{equation}
This section $\widehat{\omega}$ vanishes exactly on the diagonal $\Delta\, \subset\, X\times
X$; this is because $\omega\, \in\, \bigwedge\nolimits^2 {\mathbb V}$.
Consequently, there is a unique holomorphic isomorphism
\begin{equation}\label{P}
\Phi\, :\, (p^*_1L)\otimes (p^*_2L)\, \stackrel{\sim}{\longrightarrow}\, {\mathcal O}_{X\times X} (\Delta)
\end{equation}
such that the section $\Phi(\widehat{\omega})$, where $\widehat\omega$
is constructed in \eqref{som}, coincides with the holomorphic section of 
${\mathcal O}_{X\times X} (\Delta)$ given by the constant function $1$ on $X\times X$. 

The isomorphism $\Phi$ in \eqref{P} depends on $\omega$. It should be mentioned that
there is no natural isomorphism between ${\mathcal O}_{X\times X} (\Delta)$ and $(p^*_1L)\otimes (p^*_2L)$.

Observe that $\omega$ in \eqref{e5}
defines a symplectic structure on $\mathbb V$, in particular, $\omega$ produces
an isomorphism
\begin{equation}\label{l3i}
{\mathbb V}^*\,\stackrel{\sim}{\longrightarrow}\, {\mathbb V}\, .
\end{equation}
The action of $\text{SL}({\mathbb V})$ on $\bigwedge\nolimits^2 {\mathbb V}$, given by the action of
$\text{SL}({\mathbb V})$ on $\mathbb V$, is the trivial action, and hence this action
fixes the element $\omega$ in \eqref{e5}. We observe that $\omega$ produces a holomorphic isomorphism
\begin{equation}\label{l3}
\Psi\, :\, L^{-1}\, \longrightarrow\, {\mathbb L}\, ,
\end{equation}
where $\mathbb L$ was constructed in \eqref{l2}. Indeed, from the short exact sequence of
holomorphic vector bundles
\begin{equation}\label{exl}
0\, \longrightarrow\, {\mathbb L}\, \longrightarrow\, X\times{\mathbb V} \, \longrightarrow\, L
\, \longrightarrow\,0
\end{equation}
we conclude that ${\mathbb L}\otimes L\,=\, X\times \bigwedge^2 {\mathbb V}$. Therefore, the 
element $\omega\,\in\, \bigwedge^2 {\mathbb V}$ gives a nowhere vanishing holomorphic section 
of ${\mathbb L}\otimes L$. This section produces the isomorphism $\Psi$ in \eqref{l3}. Observe 
that the diagonal action of ${\rm SL}({\mathbb V})$ on $X\times{\mathbb V}$ preserves the subbundle 
$\mathbb L$ in \eqref{exl}; more precisely, the homomorphisms in the exact sequence \eqref{exl} are ${\rm 
SL}({\mathbb V})$--equivariant. The isomorphism $\Psi$ in \eqref{l3} is evidently ${\rm 
SL}({\mathbb V})$--equivariant. It should be mentioned that $\Psi$ depends on $\omega$.

Combining $\Phi$ in \eqref{P} with $\Psi$ in \eqref{l3}, we have
\begin{equation}\label{e6}
(p^*_1{\mathbb L})\otimes (p^*_2{\mathbb L})\,=\,
(p^*_1L^{-1})\otimes (p^*_2L^{-1})\,=\, (p^*_1{\mathbb L})\otimes (p^*_2L^{-1})
\,=\, {\mathcal O}_{X\times X} (-\Delta)\, .
\end{equation}

Using the identification of $X$ with $\Delta$, the Poincar\'e adjunction formula gives that
${\mathcal O}_{X\times X} (\Delta)\vert_\Delta\,=\, TX$, where $TX$ is the holomorphic tangent
bundle of $X$; this isomorphism ${\mathcal O}_{X\times X} (\Delta)\vert_\Delta\,=\, TX$
is the dual of the isomorphism in \eqref{e0} for $k\,=\, 1$.
Therefore, restricting the isomorphisms in \eqref{e6} to $\Delta\, \subset\, X\times X$,
we have a holomorphic isomorphism
\begin{equation}\label{e-3}
\Phi_0\, :\, L^2\,=\, ({\mathbb L}^*)^{\otimes 2}\, \stackrel{\sim}{\longrightarrow}\, TX\, .
\end{equation}

Using the isomorphisms in \eqref{e6}, we have
$$
(p^*_2 L^n)\otimes{\mathcal O}_{X\times X} (-(k+1)\Delta)\,=\,
(p^*_1L^{-k-1})\otimes (p^*_2 L^{n-k-1})\,=\, (p^*_1{\mathbb L}^{k+1})\otimes (p^*_2 L^{n-k-1})\, .
$$
Hence by the projection formula, \cite[p.~426, A4]{Ha},
\begin{equation}\label{e7}
R^ip_{1*}((p^*_2 L^n)\otimes{\mathcal O}_{X\times X} (-(k+1)\Delta))\,=\,
R^ip_{1*}((p^*_1{\mathbb L}^{k+1})\otimes (p^*_2 L^{n-k-1}))
\end{equation}
$$
=\, {\mathbb L}^{k+1}\otimes R^ip_{1*}(p^*_2 L^{n-k-1})
$$
for $i\,=\,0,\, 1$.

We shall break all possible $(n,\, k)$ into three cases.

\subsection{The case of $n\, <\, 0$}

In \eqref{e4}, we have $R^0p_{1*} p^*_2 L^n\,=\, 0$, because $n\, <\, 0$. Consequently,
the short exact sequence in \eqref{e4} becomes
\begin{equation}\label{e4a}
0\, \longrightarrow\, J^k(L^n)\, \longrightarrow\, R^1p_{1*}((p^*_2 L^n)\otimes{\mathcal
O}_{X\times X} (-(k+1)\Delta)) \, \longrightarrow\, R^1p_{1*} p^*_2 L^n \, \longrightarrow\, 0\, .
\end{equation}
Now using \eqref{e7}, the exact sequence in \eqref{e4a} becomes the short exact sequence
\begin{equation}\label{e8}
0 \, \longrightarrow\, J^k(L^n)\, \longrightarrow\,
{\mathbb L}^{k+1}\otimes R^1p_{1*}(p^*_2 L^{n-k-1})
\, \stackrel{\gamma_1}{\longrightarrow}\, R^1p_{1*} p^*_2 L^n \, \longrightarrow\, 0\, .
\end{equation}

Invoking Serre duality and using the isomorphism $\Phi_0$ in \eqref{e-3},
$$
H^1(X,\, L^n)\,=\, H^0(X,\, L^{-n-2})^*\, =\, \text{Sym}^{-n-2}({\mathbb V})^*\, ,
$$
where $\text{Sym}^{-n-2}({\mathbb V})$ denotes:
\begin{itemize}
\item the $(-n-2)$--th symmetric power of $\mathbb V$, if $-n-2\, >\, 0$,

\item $\mathbb C$, if $-n\,=\,2$, and

\item $0$, if $-n\,=\, 1$.
\end{itemize}
Therefore, $R^1p_{1*} p^*_2 L^n$ in \eqref{e8} is the trivial holomorphic vector
bundle
$$
X\times \text{Sym}^{-n-2}({\mathbb V})^*\, \longrightarrow\, X\, 
$$
(see \eqref{ci}).

For notational convenience, the trivial holomorphic vector bundle 
$X\times\text{Sym}^{j}({\mathbb V})$ over $X$ with fiber $\text{Sym}^{j}({\mathbb V})$ will be 
denoted by ${\mathbb V}_j$. The trivial line bundle ${\mathcal O}_X$ will be denoted by 
${\mathbb V}_0$. So we have
\begin{equation}\label{e9}
R^1p_{1*} p^*_2 L^n\,=\, ({\mathbb V}_{-n-2})^* \,=\, {\mathbb V}^*_{-n-2}\, .
\end{equation}

Similarly, we have
\begin{equation}\label{sd}
R^1p_{1*}(p^*_2 L^{n-k-1})\,=\, X\times \text{Sym}^{k-n-1}({\mathbb V})^*
\,=\, {\mathbb V}^*_{k-n-1} \, \longrightarrow\, X\, .
\end{equation}
Hence ${\mathbb L}^{k+1}\otimes R^1p_{1*}(p^*_2 L^{n-k-1})$ in \eqref{e8} has the following description:
\begin{equation}\label{e10}
{\mathbb L}^{k+1}\otimes
R^1p_{1*}(p^*_2 L^{n-k-1})\,=\, {\mathbb L}^{k+1}\otimes {\mathbb V}^*_{k-n-1}\, .
\end{equation}

Using \eqref{e9} and \eqref{e10}, the homomorphism $\gamma_1$ in \eqref{e8} is a homomorphism
\begin{equation}\label{e11}
\gamma_1\, :\, {\mathbb L}^{k+1}\otimes {\mathbb V}^*_{k-n-1} \, \longrightarrow\,
{\mathbb V}^*_{-n-2}\, .
\end{equation}
The homomorphism $\gamma_1$ in \eqref{e11} is simply the contraction of elements of
$\mathbb V$ by ${\mathbb V}^*$ (recall that ${\mathbb L}\,\subset\, X\times
\mathbb V$). More precisely, consider the contraction homomorphism
$$
{\mathbb V}\otimes ({\mathbb V}^*)^{\otimes j}\, \longrightarrow\, ({\mathbb V}^*)^{\otimes (j-1)}\, .
$$
Since ${\mathbb L}\,\subset\, {\mathbb V}_1$ (see \eqref{exl}), this contraction produces a
holomorphic homomorphism of vector bundles
$$
{\mathbb L}\otimes{\mathbb V}^*_j \, \longrightarrow\,{\mathbb V}^*_{j-1}\, .
$$
The homomorphism $\gamma_1$ in \eqref{e11} is $(k+1)$--times iteration of this homomorphism
starting with $j\,=\, k-n-1$.

We shall compute the kernel of the homomorphism $\gamma_1$ in \eqref{e11}.

Taking the dual of the short exact sequence in \eqref{exl}, we have
$$
L^*\, \subset\, X\times {\mathbb V}^*\,=\, {\mathbb V}^*_1\, .
$$
Hence, using the isomorphism $\Psi$ in \eqref{l3}, we have
\begin{equation}\label{e12}
{\mathbb L}\, =\, L^*\, \subset\, {\mathbb V}^*_1\, .
\end{equation}
For the natural duality paring ${\mathbb V}_1\otimes {\mathbb V}^*_1\, \longrightarrow\,
{\mathcal O}_X$,
the annihilator of the subbundle $\mathbb L\, \subset\, {\mathbb V}_1$ in \eqref{exl} is
the subbundle ${\mathbb L}\, \subset\, {\mathbb V}^*_1$ in \eqref{e12}. Therefore, from the above
description of the homomorphism $\gamma_1$ in \eqref{e11} we conclude that the subbundle
$$
{\mathbb L}^{k+1}\otimes ({\mathbb V}^*_{k-n-1}) \, \supset\, {\mathbb L}^{k+1}\otimes
({\mathbb L}^{-n-1}\otimes{\mathbb V}^*_k)\,=\, {\mathbb L}^{k-n}\otimes{\mathbb V}^*_k
$$
is contained in the kernel of the surjective homomorphism $\gamma_1$;
note that ${\mathbb L}^{-n-1}\otimes{\mathbb V}^*_k$
is realized as a subbundle of ${\mathbb V}^*_{k-n-1}$ using the inclusion
${\mathbb L}\, \subset\, {\mathbb V}^*_1$ in \eqref{e12}. From \eqref{e8} we know
that $\gamma_1$ is surjective. Since
$$
\text{rank}({\mathbb L}^{k-n}\otimes{\mathbb V}^*_k)\,=\, k+1\,=\, (k-n)-(-n-1)\,=\,
\text{rank}({\mathbb L}^{k+1}\otimes {\mathbb V}^*_{k-n-1})-\text{rank}({\mathbb V}^*_{-n-2})\, ,
$$
we now conclude that
$$
{\mathbb L}^{k-n}\otimes{\mathbb V}^*_k\,=\, \text{kernel}(\gamma_1)\, .
$$
Therefore, from \eqref{e8} it follows that
\begin{equation}\label{e13}
J^k(L^n)\,=\, \text{kernel}(\gamma_1)\,=\, {\mathbb L}^{k-n}\otimes{\mathbb V}^*_k\, .
\end{equation}
Now using the isomorphism between ${\mathbb V}_k$ and
${\mathbb V}^*_k$ given by the isomorphism in \eqref{l3i}, from \eqref{e13} we deduce that
$$
J^k(L^n)\,=\, {\mathbb L}^{k-n}\otimes {\mathbb V}_k\, .
$$
Moreover, using the isomorphism $\Psi$ in \eqref{l3} we have 
\begin{equation}\label{e16}
J^k(L^n)\,=\, L^{n-k}\otimes {\mathbb V}_k\, .
\end{equation}

Recall that all the homomorphisms in \eqref{e4} are $\text{SL}({\mathbb V})$--equivariant. 
From the construction of the isomorphism in \eqref{e13} between $J^k(L^n)$ and ${\mathbb 
L}^{k-n} \otimes{\mathbb V}^*_k$ it is evident that it is $\text{SL}({\mathbb 
V})$--equivariant. Therefore, the isomorphism in \eqref{e16} is $\text{SL}({\mathbb 
V})$--equivariant.

Fix $k\, \geq\, 1$ an integer. Using \eqref{e16}, the homomorphism $\phi_k$ in \eqref{e2}
is a projection
\begin{equation}\label{a1}
\phi_k\, :\, L^{n-k}\otimes {\mathbb V}_k\,=\, J^k(L^n)
\, \longrightarrow\, J^{k-1}(L^n)\,=\,
L^{n-k+1}\otimes {\mathbb V}_{k-1}\, .
\end{equation}
We shall describe the homomorphism in \eqref{a1}.

Recall that ${\mathbb V}_j\,=\, \text{Sym}^j({\mathbb V}_1)\,=\,
\text{Sym}^j(X\times{\mathbb V})$, and the line bundle $L$
is a quotient of ${\mathbb V}_1$ (see \eqref{exl}). Therefore, we have a natural projection
\begin{equation}\label{vk0}
\varpi^0_k\, :\, {\mathbb V}_k\, \longrightarrow\,L\otimes {\mathbb V}_{k-1}\, .
\end{equation}
To describe this $\varpi^0_k$ explicitly, let $q_{L}\, :\, {\mathbb V}_1\, \longrightarrow\,
L$ be the quotient map in \eqref{exl}. Consider the homomorphism
$$
q_L\otimes \text{Id}_{{\mathbb V}^{\otimes(k-1)}_1}\, :\, {\mathbb V}^{\otimes k}_1\,
\longrightarrow\, L\otimes{\mathbb V}^{\otimes (k-1)}_1\, .
$$
Note that $q_L\otimes\text{Id}_{{\mathbb V}^{\otimes(k-1)}_1}$ sends the subbundle 
$$\text{Sym}^k({\mathbb V}_1)\, \subset\, {\mathbb V}^{\otimes k}_1$$ to the subbundle $L\otimes 
{\mathbb V}_{k-1}\, \subset\, L\otimes{\mathbb V}^{\otimes (k-1)}_1$. The homomorphism 
$\varpi^0_k$ in \eqref{vk0} is this restriction of $q_L\otimes\text{Id}_{{\mathbb 
V}^{\otimes(k-1)}_1}$.

\begin{proposition}\label{prop2}
The homomorphism $\phi_k$ in \eqref{a1} coincides with ${\rm Id}_{L^{n-k}}\otimes\varpi^0_k$, where
$\varpi^0_k$ is the homomorphism in \eqref{vk0}.
\end{proposition}

\begin{proof}
Consider the short exact sequence in \eqref{e4a}. We have the commutative diagram of
homomorphisms
\begin{equation}\label{e4b}
\begin{matrix}
&& 0 &&0\\
&& \Big\downarrow &&\big\downarrow\\
&& L^n\otimes K^{\otimes k}_X &= & L^n\otimes K^{\otimes k}_X\\
&& \Big\downarrow &&\big\downarrow\\
0 & \longrightarrow & J^k(L^n) & \longrightarrow & R^1p_{1*}((p^*_2 L^n)\otimes{\mathcal
O}_{X\times X} (-(k+1)\Delta)) & \longrightarrow & R^1p_{1*} p^*_2 L^n & \longrightarrow & 0\\
&&\Big\downarrow &&~\,~\,\Big\downarrow\mu &&\Vert\\
0 & \longrightarrow & J^{k-1}(L^n) & \longrightarrow & R^1p_{1*}((p^*_2 L^n)\otimes{\mathcal
O}_{X\times X} (-k\Delta)) & \longrightarrow & R^1p_{1*} p^*_2 L^n & \longrightarrow & 0\\
&& \Big\downarrow &&\big\downarrow\\
&& 0&& 0
\end{matrix}
\end{equation}
where the bottom exact sequence is the one in \eqref{e4a} with $k$ substituted by $k-1$,
the top exact sequence is the one in \eqref{e4a}, the left-vertical exact sequence is the one in
\eqref{e2} for $W\, =\, L^n$, and the other vertical exact sequence is the long
exact sequence of direct images, under $p_1$, for the short exact sequence
$$
0\, \longrightarrow\, (p^*_2 L^n)\otimes{\mathcal O}_{X\times X} (-(k+1)\Delta)\,
\longrightarrow\, (p^*_2 L^n)\otimes{\mathcal O}_{X\times X} (-k\Delta)
\, \longrightarrow\, (L^n\otimes K^{\otimes k}_X)\vert_\Delta \, \longrightarrow\, 0
$$
of sheaves on $X\times X$ obtained by tensoring with $p^*_2 L^n$ the short exact
sequence given by \eqref{e0n}.

Using \eqref{e7} and \eqref{e10} we have
$$
R^1p_{1*}((p^*_2 L^n)\otimes{\mathcal O}_{X\times X} (-(k+1)\Delta)) \,=\, 
{\mathbb L}^{k+1}\otimes {\mathbb V}^*_{k-n-1}\, .
$$
Replacing $k$ by $k-1$, we have
$$
R^1p_{1*}((p^*_2 L^n)\otimes{\mathcal O}_{X\times X} (-k\Delta)) \,=\, 
{\mathbb L}^{k}\otimes {\mathbb V}^*_{k-n-2}\, .
$$
So the homomorphism $\mu$ in \eqref{e4b} becomes a homomorphism
$$
\mu\,:\, {\mathbb L}^{k+1}\otimes {\mathbb V}^*_{k-n-1}\, \longrightarrow\,
{\mathbb L}^{k}\otimes {\mathbb V}^*_{k-n-2}\, .
$$
This homomorphism is simply the contraction of elements of ${\mathbb V}^*$ by elements
of ${\mathbb L}\, \subset\, \mathbb V$.
The proposition follows from this and the commutative diagram in \eqref{e4b}.
\end{proof}

\subsection{The case of $n\, \geq\, k$}

Note that $n\, \geq\, 0$, because $n\geq\, k$.
Since now we have $n-k-1\, \geq\, -1$, it follows that
$$
H^1(X,\, L^{n-k-1})\,=\, 0\, ,
$$
and hence using \eqref{e7} and \eqref{ci} it is deduced that
$$
R^1p_{1*}((p^*_2 L^n)\otimes{\mathcal O}_{X\times X} (-(k+1)\Delta))\, =\, 0\, .
$$
Consequently, the exact sequence in \eqref{e4} becomes the short exact sequence
\begin{equation}\label{s1}
0\, \longrightarrow\, R^0p_{1*}((p^*_2 L^n)\otimes{\mathcal O}_{X\times X} (-(k+1)\Delta))
\, \longrightarrow\, R^0p_{1*} p^*_2 L^n \, \longrightarrow\, J^k(L^n)\, \longrightarrow\, 0\, .
\end{equation}
Therefore, using \eqref{e7} and \eqref{ci}, the exact sequence in \eqref{s1} becomes the short exact sequence
\begin{equation}\label{e14}
0\, \longrightarrow\, {\mathbb L}^{k+1}\otimes
R^0p_{1*}(p^*_2 L^{n-k-1})\,=\, {\mathbb L}^{k+1}\otimes{\mathbb V}_{n-k-1}
\end{equation}
$$
\stackrel{\gamma_2}{\longrightarrow}\, X\times \text{Sym}^n({\mathbb V})
\,=:\, {\mathbb V}_n \, \longrightarrow\, J^k(L^n)
\, \longrightarrow\, 0\, .
$$

Consider the homomorphism given by the composition of the natural homomorphisms
$$
{\mathbb V}\otimes {\mathbb V}_j\, \longrightarrow\, X\times {\mathbb V}^{\otimes (j+1)}
\, \longrightarrow\, X\times \text{Sym}^{j+1}({\mathbb V})\,=\, {\mathbb V}_{j+1}\, .
$$
Since ${\mathbb L}\, \subset\, \mathbb V$ (see \eqref{exl}), this restricts to a homomorphism
$$
{\mathbb L}\otimes {\mathbb V}_j\, \longrightarrow\, {\mathbb V}_{j+1}\, .
$$
The homomorphism $\gamma_2$ in \eqref{e14} is $(k+1)$--fold iteration of this homomorphism
starting with $j\,=\,n-k-1$.

To compute the cokernel of this homomorphism $\gamma_2$, consider the surjective homomorphism
${\mathbb V}_1 \, \longrightarrow\, L$ in \eqref{exl}. It produces a surjective homomorphism
$$
\theta_2\, :\, {\mathbb V}_n \, \longrightarrow\, L^{n-k}\otimes {\mathbb V}_k\, .
$$
It is straight-forward to check that $\theta_2\circ\gamma_2\,=\, 0$. Consequently, $\theta_2$
produces a surjective homomorphism from $\text{cokernel}(\gamma_2)$ to $L^{n-k}\otimes{\mathbb V}_k$.
On the other hand,
$$
\text{rank}({\mathbb V}_n)-\text{rank}({\mathbb L}^{k+1}\otimes{\mathbb V}_{n-k-1})\,=\,
k+1\,=\, \text{rank}(L^{n-k}\otimes {\mathbb V}_k)\, .
$$
Hence we conclude that the above surjection from $\text{cokernel}(\gamma_2)$
to $L^{n-k}\otimes {\mathbb V}_k$ is also injective.

Therefore, from \eqref{e14} it follows that
\begin{equation}\label{e15}
J^k(L^n)\,=\, L^{n-k}\otimes {\mathbb V}_k\, .
\end{equation}

{}From the construction of the isomorphism in \eqref{e15} it is evident that
this isomorphism is in fact $\text{SL}({\mathbb V})$--equivariant.

Consider $k\, \geq\, 1$ an integer. Using \eqref{e15}, the homomorphism $\phi_k$ in \eqref{e2}
is a projection
\begin{equation}\label{a1p}
\phi_k\, :\, L^{n-k}\otimes {\mathbb V}_k\,=\, J^k(L^n)
\, \longrightarrow\, J^{k-1}(L^n)\,=\, L^{n-k+1}\otimes {\mathbb V}_{k-1}\, .
\end{equation}

\begin{proposition}\label{prop3}
The homomorphism $\phi_k$ in \eqref{a1p} coincides with ${\rm Id}_{L^{n-k}}\otimes
\varpi^0_k$, where $\varpi^0_k$ is the homomorphism in \eqref{vk0}.
\end{proposition}

\begin{proof}
Consider the short exact sequence of sheaves in \eqref{s1}. We have the commutative diagram of homomorphism
\begin{equation}\label{d1}
\begin{matrix}
&& &&&& 0\\
&& &&&& \Big\downarrow\\
&& 0 &&&& L^n\otimes K^{\otimes k}_X\\ 
&& \Big\downarrow &&&& \Big\downarrow\\
0 & \longrightarrow & R^0p_{1*}((p^*_2 L^n)\otimes{\mathcal O}_{X\times X} (-(k+1)\Delta))
& \longrightarrow & R^0p_{1*} p^*_2 L^n & \longrightarrow & J^k(L^n) & \longrightarrow & 0\\
&& ~\,~\,\Big\downarrow\mu && \Vert && \Big\downarrow\\
0 & \longrightarrow & R^0p_{1*}((p^*_2 L^n)\otimes{\mathcal O}_{X\times X} (-k\Delta))
& \longrightarrow & R^0p_{1*} p^*_2 L^n & \longrightarrow & J^{k-1}(L^n) & \longrightarrow & 0\\
&& \Big\downarrow && && \Big\downarrow\\
&& L^n\otimes K^{\otimes k}_X &&&& 0\\
&& \Big\downarrow\\
&& 0
\end{matrix}
\end{equation}
where the bottom exact row is \eqref{s1} with $k-1$ substituted in place of $k$, the top exact
row is the one in \eqref{s1}, the right vertical exact sequence is the one in \eqref{e2} and the
left vertical exact sequence is a part of the long
exact sequence of direct images, under $p_1$, for the short exact sequence
$$
0\, \longrightarrow\, (p^*_2 L^n)\otimes{\mathcal O}_{X\times X} (-(k+1)\Delta)\,
\longrightarrow\, (p^*_2 L^n)\otimes{\mathcal O}_{X\times X} (-k\Delta)
\, \longrightarrow\, (L^n\otimes K^{\otimes k}_X)\vert_\Delta \, \longrightarrow\, 0
$$
of sheaves on $X\times X$ obtained by tensoring \eqref{e0} with $p^*_2L^n$.
Using \eqref{e7}, in \eqref{e14} we saw that
$$
R^0p_{1*}((p^*_2 L^n)\otimes{\mathcal O}_{X\times X} (-(k+1)\Delta))
\,=\, {\mathbb L}^{k+1}\otimes{\mathbb V}_{n-k-1}\, ,
$$
So replacing $k$ by $k-1$, we have
$$
R^0p_{1*}((p^*_2 L^n)\otimes{\mathcal O}_{X\times X} (-k\Delta))
\,=\, {\mathbb L}^{k}\otimes{\mathbb V}_{n-k}\, .
$$
Using these isomorphisms, the injective homomorphism
$$
\mu\, :\,{\mathbb L}^{k+1}\otimes{\mathbb V}_{n-k-1}\,=\,
{\mathbb L}^{k+1}\otimes\text{Sym}^{n-k-1}({\mathbb V}_1)\, \longrightarrow\,
{\mathbb L}^{k}\otimes\text{Sym}^{n-k}({\mathbb V}_1)
$$
in \eqref{d1} is the natural homomorphism induced by the inclusion of ${\mathbb L}$ in
${\mathbb V}_1$. Now the proposition follows from \eqref{d1}.
\end{proof}

The remaining case will now be considered.

\subsection{The case of $k\, > \, n\, \geq\, 0$}

Since $n\, \geq\, 0$, it follows that $H^1(X, \, L^n)\,=\, 0$. Hence in \eqref{e4}
we have
$$
R^1p_{1*} p^*_2 L^n\, =\, 0\, ;
$$
see \eqref{ci}. Since $n-k-1\, <\, 0$, we have $H^0(X,\, L^{n-k-1})\,=\, 0$.
Therefore, from \eqref{e7} and \eqref{ci} we conclude that
$$
R^0p_{1*}((p^*_2 L^n)\otimes{\mathcal O}_{X\times X} (-(k+1)\Delta))\,=\, 0\, .
$$
Consequently, the exact sequence in \eqref{e4} becomes the short exact sequence
\begin{equation}\label{s13}
0\, \longrightarrow\, R^0p_{1*} p^*_2 L^n \, \longrightarrow\, J^k(L^n)
\, \longrightarrow\, R^1p_{1*}((p^*_2 L^n)\otimes{\mathcal O}_{X\times X} (-(k+1)\Delta))
\, \longrightarrow\, 0\, .
\end{equation}
Now using \eqref{e7} and \eqref{ci} together with the projection formula, the exact
sequence in \eqref{s13} becomes the exact sequence
\begin{equation}\label{e17}
0 \, \longrightarrow\, {\mathbb V}_n\, \stackrel{\gamma_3}{\longrightarrow}\, J^k(L^n) \,
\stackrel{\theta_3}{\longrightarrow}\,{\mathbb L}^{k+1}\otimes R^1p_{1*} (p^*_2 L^{n-k-1})
\, \longrightarrow\, 0\,.
\end{equation}
The homomorphisms $\gamma_3$ and and $\theta_3$ in \eqref{e17} are clearly
${\rm SL}({\mathbb V})$--equivariant.

\begin{lemma}\label{lem1}
The short exact sequence in \eqref{e17} has a canonical ${\rm SL}({\mathbb V})$--equivariant
holomorphic splitting.
\end{lemma}

\begin{proof}
Since $k\,>\, n$, there is a natural projection
$$
J^k(L^n)\, \longrightarrow\, J^n(L^n)
$$
(see \eqref{e2}). Now $J^n(L^n)\,=\, {\mathbb V}_n$ by \eqref{e15}, so this projection defines a
projection
\begin{equation}\label{e18}
\beta\, :\, J^k(L^n)\, \longrightarrow\, {\mathbb V}_n\, .
\end{equation}
It is straight-forward to check that $\beta\circ\gamma_3\,=\, \text{Id}_{{\mathbb V}_n}$, where
$\gamma_3$ and $\beta$ are constructed in \eqref{e17} and \eqref{e18} respectively. This produces
an ${\rm SL}({\mathbb V})$--equivariant holomorphic splitting
$$
J^k(L^n) \,=\, \text{image}(\gamma_3) \oplus \text{kernel}(\beta)
$$
of the short exact sequence in \eqref{e17}.
\end{proof}

Lemma \ref{lem1} gives a holomorphic decomposition
\begin{equation}\label{e19}
J^k(L^n)\,=\, {\mathbb V}_n\oplus ({\mathbb L}^{k+1}\otimes R^1p_{1*}(p^*_2 L^{n-k-1})
\end{equation}
which is compatible with the actions of $\text{SL}({\mathbb V})$. Now,
$R^1p_{1*}(p^*_2 L^{n-k-1})\,=\, {\mathbb V}^*_{k-n-1}$ (see \eqref{sd}), and
${\mathbb V}^*_{k-n-1}\,=\, {\mathbb V}_{k-n-1}$ which is
constructed using the isomorphism in \eqref{l3i}. Therefore, the decomposition
in \eqref{e19} can be reformulated as
\begin{equation}\label{e20}
J^k(L^n)\,=\, {\mathbb V}_n\oplus ({\mathbb L}^{k+1}\otimes {\mathbb V}_{k-n-1})\, =\,
{\mathbb V}_n\oplus (L^{-(k+1)}\otimes {\mathbb V}_{k-n-1})
\end{equation}
using $\Psi$ in \eqref{l3}. As before, the isomorphisms in \eqref{e20} are
compatible with the actions of $\text{SL}({\mathbb V})$.

If $k-1\, >\, n\, \geq\, 0$, then from \eqref{e20} we know that
$$
J^{k-1}(L^n)\,=\, {\mathbb V}_n\oplus ({\mathbb L}^{k}\otimes {\mathbb V}_{k-n-2})\, =\,
{\mathbb V}_n\oplus (L^{-k}\otimes {\mathbb V}_{k-n-2})\, ,
$$
and if $k-1\, =\, n\, \geq\, 0$, then from \eqref{e15} we have
$$
J^{k-1}(L^n)\,=\, {\mathbb V}_{k-1}\, .
$$

The proof of the following proposition is similar to the proofs of Proposition \ref{prop2}
and Proposition \ref{prop3}.

\begin{proposition}\label{prop4}
If $k-1\, >\, n\, \geq\, 0$, then the projection $J^k(L^n)\, \longrightarrow\, J^{k-1}(L^n)$
in \eqref{e2} is the homomorphism
$$
{\mathbb V}_n\oplus (L^{-(k+1)}\otimes {\mathbb V}_{k-n-1})\, \longrightarrow\,
{\mathbb V}_n\oplus (L^{-k}\otimes {\mathbb V}_{k-n-2})
$$
which is the direct sum of the identity map of ${\mathbb V}_n$ and the natural contraction map
$L^{-(k+1)}\otimes {\mathbb V}_{k-n-1}\, \longrightarrow\,
L^{-k}\otimes {\mathbb V}_{k-n-2}$.

If $k-1\, =\, n\, \geq\, 0$, then the projection $J^k(L^n)\, \longrightarrow\, J^{k-1}(L^n)$
in \eqref{e2} is the homomorphism
$$
{\mathbb V}_n\oplus (L^{-(k+1)}\otimes {\mathbb V}_{k-n-1})\, \longrightarrow\,
{\mathbb V}_n
$$
which is the direct sum of the identity map of ${\mathbb V}_n$ and the zero homomorphism of
$L^{-(k+1)}\otimes {\mathbb V}_{k-n-1}\,=\, L^{-(k+1)}$ (recall that ${\mathbb V}_0\,=\,
{\mathcal O}_X$).
\end{proposition}

\subsection{The jet bundle}

Combining \eqref{e16}, \eqref{e15} and \eqref{e20}, we have the following:

\begin{theorem}\label{thm1}
If $n\, <\, 0$ or $n\, \geq\, k$, then
$$
J^k(L^n)\,=\, L^{n-k}\otimes {\mathbb V}_k\, .
$$
If $k\, > \, n\, \geq\, 0$, then
$$
J^k(L^n)\,=\, {\mathbb V}_n\oplus (L^{-(k+1)}\otimes {\mathbb V}_{k-n-1})\, .
$$
Both the isomorphisms are ${\rm SL}({\mathbb V})$--equivariant.
\end{theorem}

\section{Holomorphic connections on vector bundles}\label{s4}

In this section we construct and study natural differential holomorphic operators between 
natural holomorphic vector bundles on a Riemann surface equipped with a projective structure.

\subsection{Jets of Riemann surfaces with a projective structure}

Let $M$ be a compact connected Riemann surface
of genus $g$. A holomorphic coordinate chart on $M$ is a pair
of the form $(U,\, \varphi)$, where $U\, \subset\, M$ is an open subset and $\varphi\,
:\, U\, \longrightarrow\, {\mathbb P}({\mathbb V})\,=\, X$ is a holomorphic embedding. A holomorphic
coordinate atlas on $M$ is a collection of coordinate charts $\{(U_i,\, \varphi_i)\}_{i\in I}$
such that $M\,=\, \bigcup_{i\in I} U_i$. A projective structure on $M$ is given by a holomorphic 
coordinate atlas $\{(U_i,\, \varphi_i)\}_{i\in I}$ such that for all 
pairs $(i, j) \,\in\, I\times I$ for which $U_i\bigcap U_j\, \not=\, \emptyset$, there exists 
an element $\tau_{j,i}\, \in\, \text{Aut}({\mathbb P}({\mathbb V}))\,=\, \text{PGL}({\mathbb V})$
such that $\varphi_j\circ\varphi^{-1}_i$ is the restriction of $\tau_{j,i}$ to
$\varphi_i(U_i\bigcap U_j)$. A holomorphic coordinate atlas satisfying this condition is called
a projective atlas. Two such projective atlases $\{(U_i,\, \varphi_i)\}_{i\in I}$ and
$\{(U_i,\, \varphi_i)\}_{i\in J}$ are called equivalent if
their union $\{(U_i,\, \varphi_i)\}_{i\in I\cup J}$ is again a
projective atlas. A \textit{projective structure} on $M$
is an equivalence class of projective atlases (see \cite{Gu}).

Giving a projective structure on $M$ is equivalent to giving a holomorphic Cartan geometry on
$M$ for the pair of groups $(\text{PSL}(2,{\mathbb C}),\, B)$, where $B\, \subset\,
\text{PSL}(2,{\mathbb C})$ is the Borel subgroup
$$
B\, :=\, \Big\{
\begin{pmatrix}
a & b\\
c &d
\end{pmatrix}\, \in\, \text{PSL}(2,{\mathbb C})\, \mid\, d\, =\, 0\Big\}\, ;
$$
see \cite{Sh} for Cartan geometry.

Note that the collection $\{\tau_{j,i}\}$ in the definition of a projective atlas forms a 
$1$--cocycle with values in $\text{PGL}({\mathbb V})$. The corresponding cohomology class in 
$H^1(M,\, \text{PGL}({\mathbb V}))$ depends only on the projective structure and can be lifted 
to $H^1(M,\, \text{SL}({\mathbb V}))$. Moreover the space of such lifts is in bijection with 
the theta characteristics of $M$ (see \cite{Gu}). We recall that the theta characteristic on 
$M$ is a holomorphic line bundle $K^{1/2}_M$ on $M$ together with a holomorphic isomorphism of 
$K^{1/2}_M\otimes K^{1/2}_M$ with the holomorphic cotangent bundle $K_M$.

We fix a theta characteristic on $M$. So a projective structure on $M$ is given by a
coordinate atlas $\{(U_i,\, \varphi_i)\}_{i\in I}$ together with an element
$$
\tau_{j,i}\, \in\, \text{SL}({\mathbb V})
$$
for every ordered pair $(i,\, j)\,\in\, I\times I$ such that $U_i\bigcap U_j\, \not=\,
\emptyset$ such that the following four conditions hold:
\begin{enumerate}
\item $\varphi_j\circ\varphi^{-1}_i$ is the restriction to $\varphi_i(U_i\bigcap U_j)$ of the
automorphism of $X$ given by $\tau_{j,i}$,

\item $\tau_{i,k}\tau_{k,j}\tau_{j,i}\,=\, \text{Id}_{\mathbb V}$ for all $i,\, j,\, k\, \in\, I$
such that $U_i\bigcap U_j\bigcap U_k\, \not=\, \emptyset$,

\item $\tau_{i,i}\,=\, \text{Id}_{\mathbb V}$ for all $i\, \in\, I$, and

\item the theta characteristic corresponding to the data is the given one.
\end{enumerate}

Two such collections of triples $\{\{(U_i,\, \varphi_i)\}_{i\in I},\, \{\tau_{i,j}\}\}$ and
$\{\{(U_i,\, \varphi_i)\}_{i\in J},\, \{\tau_{i,j}\}\}$ are called equivalent if their union
$\{\{(U_i,\, \varphi_i)\}_{i\in I\cup J},\, \{\tau_{i,j}\}\}$ is again a part of a collection
of triples satisfying the above conditions. A $\text{SL}({\mathbb V})$--\textit{projective
structure} on $M$ is an equivalence class of collection of triples satisfying the above conditions.

Giving a $\text{SL}({\mathbb V})$--projective structure on $M$
is equivalent to giving a holomorphic Cartan geometry on $M$ for the pair of
groups $(\text{SL}(2,{\mathbb C}),\, B_0)$, where $B_0\, \subset\, \text{SL}(2,{\mathbb C})$ is the
Borel subgroup
$$
B_0\, :=\, \Big\{
\begin{pmatrix}
a & b\\
c &d
\end{pmatrix}\, \in\, \text{SL}(2,{\mathbb C})\, \mid\, d\, =\, 0\Big\}\, .
$$

Let $P\,=\, \{\{(U_i,\, \varphi_i)\}_{i\in I},\, \{\tau_{i,j}\}\}$ be a $\text{SL}({\mathbb V})$--projective
structure on $M$. Recall that the collection $\{\tau_{i,j}\}$ forms a $1$--cocycle on $M$
with values in the group $\text{SL}({\mathbb V})$. Hence $\{\tau_{i,j}\}$ produce a holomorphic rank
two bundle $\mathcal V$ on $M$ equipped with a holomorphic connection
\begin{equation}\label{D}
D\, :\, {\mathcal V}\, \longrightarrow\, {\mathcal V}\otimes K_M\, ,
\end{equation}
where $K_M$ is the holomorphic cotangent bundle of $M$ (see
\cite{At2} for holomorphic connections). To construct
the pair $({\mathcal V},\, D)$ explicitly, for each $i\, \in\, I$, consider the trivial
holomorphic vector bundle $U_i\times {\mathbb V}
\, \longrightarrow\, U_i$ equipped with the
trivial connection. For any ordered pair $(i,\, j)\, \in\,I\times I$ with $U_i\bigcap U_j\,
\not=\, \emptyset$, glue $U_i\times {\mathbb V}$ and $U_j\times {\mathbb V}$ over
$U_i\bigcap U_j$ using the
automorphism of $(U_i\bigcap U_j)\times \mathbb V$
that sends any $$(y,\, v) \, \in\, (U_i\bigcap U_j)
\times{\mathbb V}\, \subset\, U_j\times{\mathbb V}$$
to $(y,\, \tau_{i,j}(v)) \, \in\, (U_i\bigcap U_j)\times {\mathbb V}\, \subset\, 
U_i\times{\mathbb V}$. This automorphism of $(U_i\bigcap U_j)\times \mathbb V$ is holomorphic, 
and it evidently preserves the trivial connection. Therefore, this gluing operation
produces a holomorphic vector bundle on $M$, which we shall denote by $\mathcal V$,
and a holomorphic connection on $\mathcal V$, which we shall denote by $D$.

The diagonal action of $\text{SL}({\mathbb V})$ on $X\times{\mathbb V}$ preserves the
subbundle $\mathbb L$ in \eqref{exl}. Therefore, the subbundles
$\varphi_i(U_i)\times {\mathbb L}\, \subset\, \varphi(U_i)\times\mathbb V$ patch together
compatibly to produce a holomorphic line subbundle of the vector bundle $\mathcal V$ on $M$
constructed above. This line subbundle of $\mathcal V$ will be denoted by $\mathcal L$. Let
\begin{equation}\label{bl}
{\mathbf L}\, :=\, {\mathcal V}/\mathcal L
\end{equation}
be the quotient bundle.

Recall that the isomorphism $\Psi$ in \eqref{l3} and the isomorphism $\Phi_0$
in \eqref{e-3} are both ${\rm SL}({\mathbb V})$--equivariant. However, it should be
clarified that a nonzero element $\omega\, \in\, \bigwedge\nolimits^2 {\mathbb V}$
(see \eqref{e5}) was used in their construction. Since $\Psi$ and $\Phi_0$ are
${\rm SL}({\mathbb V})$--equivariant, they produce isomorphisms
\begin{equation}\label{l4}
{\mathbf L}\,=\, {\mathcal L}^* \ \ \text{ and }
{\mathbf L}^{\otimes 2}\,=\, ({\mathcal L}^*)^{\otimes 2}\, = \, TM
\end{equation}
respectively. Recall that we fixed a theta characteristic on $M$. The fourth condition in the
definition of a projective structure says that the theta characteristic $\mathcal L$ 
coincides with the chosen one.

For any $j\, \geq\, 1$, the vector bundle $\text{Sym}^j({\mathcal V})$ will be denoted by 
${\mathcal V}_j$. Also, ${\mathcal V}_0$ will denote the trivial line bundle ${\mathcal O}_M$. 
The holomorphic connection $D$ on $\mathcal V$ constructed in \eqref{D}
induces a holomorphic connection on ${\mathcal 
V}_j$ for every $j\, \geq\, 1$; let
\begin{equation}\label{D2}
D_j\, :\, {\mathcal V}_j\, \longrightarrow\, {\mathcal V}_j\otimes K_M
\end{equation}
be this connection given by $D$.
The trivial connection on ${\mathcal V}_0\,=\, {\mathcal O}_M$ will be denoted by $D_0$. The 
connections $D_m$ and $D_n$ on ${\mathcal V}_m$ and ${\mathcal V}_n$ respectively together 
produce a holomorphic connection on ${\mathcal V}_m\otimes {\mathcal V}_n$.

\begin{proposition}\label{prop1}
Let $M$ be a compact connected Riemann surface equipped with a projective structure $P$. Take integers
$m\, >\, n\, \geq\, 0$. Then ${\mathcal V}_m\otimes {\mathcal V}_n$ has a canonical decomposition
$$
{\mathcal V}_m\otimes {\mathcal V}_n\,=\,\bigoplus_{i=0}^n {\mathcal V}_{m+n-2i}\, .
$$
The above isomorphism takes the connection on ${\mathcal V}_m\otimes {\mathcal V}_n$ induced
by $D_m$ and $D_n$ to the connection $\bigoplus_{i=0}^n D_{m+n-2i}$ on
$\bigoplus_{i=0}^n {\mathcal V}_{m+n-2i}$.
\end{proposition}

\begin{proof}
First consider the case of $X\,=\, {\mathbb P}({\mathbb V})$ investigated in Section \ref{se2}. The
$\text{SL}({\mathbb V})$--module $\text{Sym}^{j}({\mathbb V})$ is irreducible for every $j\, \geq\, 0$
\cite[p.~150]{FH}, and furthermore, for $m\, >\, n\, \geq\, 0$, the $\text{SL}({\mathbb V})$--module
$\text{Sym}^{m}({\mathbb V})\otimes \text{Sym}^{n}({\mathbb V})$ decomposes as
\begin{equation}\label{td}
\text{Sym}^{m}({\mathbb V})\otimes \text{Sym}^{n}({\mathbb V})\,=\,
\bigoplus_{i=0}^n \text{Sym}^{m+n-2i}({\mathbb V})
\end{equation}
\cite[p.~151, Ex. 11.11]{FH}.

Take a theta characteristic on $M$.
Choose data $\{\{(U_i,\, \varphi_i)\}_{i\in I},\, \{\tau_{i,j}\}\}$ representing a $\text{SL}({\mathbb 
V})$--projective structure associated to $P$. On each $U_i$ we have the 
holomorphic vector bundles $\varphi^*_i {\mathbb V}_j\,=\, U_i\times \text{Sym}^j({\mathbb V})$ equipped with the 
trivial connection. Patching these together using $\{\tau_{i,j}\}$ as the transition functions we construct
a holomorphic connection $D_j$ on the holomorphic vector bundle ${\mathcal 
V}_j$ for every $j\, \geq\, 0$. This connection $D_j$ on ${\mathcal V}_j$ coincides with the connection
constructed in \eqref{D2}.

Now it can be shown that the isomorphism in 
\eqref{td} over each $U_i$ patch together compatibly to give a global holomorphic isomorphism
$$
{\mathcal V}_m\otimes {\mathcal V}_n\,\stackrel{\sim}{\longrightarrow}\,\bigoplus_{i=0}^n {\mathcal V}_{m+n-2i}
$$
over $M$. Indeed, this is an immediate consequence of the fact that the isomorphism in \eqref{td} intertwines
the actions of $\text{SL}({\mathbb V})$. This isomorphism evidently takes the connection on ${\mathcal V}_m
\otimes {\mathcal V}_n$ induced by $D_m$ and $D_n$ to the connection $\bigoplus_{i=0}^n D_{m+n-2i}$ on
$\bigoplus_{i=0}^n {\mathcal V}_{m+n-2i}$. This completes the proof.
\end{proof}

Henceforth, we shall assume that $\text{genus}(M)\,=\, g\, \, \geq\, 1$.
Since all the isomorphisms in Theorem \ref{thm1} are 
${\rm SL}({\mathbb V})$--equivariant, the following corollary is deduced from
Theorem \ref{thm1} in a straight-forward manner.

\begin{corollary}\label{cor1}
Let $M$ be a connected Riemann surface equipped with a projective structure $P$.
If $n\, <\, 0$ or $n\, \geq\, k$, then there is a canonical isomorphism
$$
J^k({\mathbf L}^n)\,=\, {\mathbf L}^{n-k}\otimes {\mathcal V}_k\, .
$$
If $k\, > \, n\, \geq\, 0$, then there is a canonical isomorphism
$$
J^k({\mathbf L}^n)\,=\, {\mathcal V}_n\oplus ({\mathbf L}^{-(k+1)}\otimes {\mathcal V}_{k-n-1})\, .
$$
\end{corollary}

\begin{remark}\label{rem1}
Note that Proposition \ref{prop1}, Corollary \ref{cor1} and \eqref{l4} together give a description
of any jet bundles of the form $J^k({\mathbf L}^a)\otimes J^\ell({\mathbf L}^b)$ as a direct sum of vector bundles
of the form ${\mathcal V}^p\otimes {\mathbf L}^q$.
\end{remark}

\subsection{Jets of vector bundles with a holomorphic connection}\label{se4.2}

Let $M$ be a compact connected Riemann surface. Let $E$ be a holomorphic vector bundle on $M$
equipped with a holomorphic connection $D_E$. Take a holomorphic vector bundle $W$ on $M$.
We will construct a holomorphic homomorphism from $J^k(E\otimes W)$ to $E\otimes J^k(W)$ for all
$k\, \geq\, 0$. To construct this homomorphism, let $f_1,\, f_2\, :\, M\times M\, \longrightarrow\,
M$ be the projections to the first and second factors respectively. The diagonal divisor
$$
\{(x,\, x)\, \mid\, x\, \in\, M\}\, \subset\, M\times M
$$
will be denoted by $\Delta_M$. Recall from \eqref{e1} that
\begin{equation}\label{jt}
J^k(E\otimes W)\, :=\, f_{1*}((f^*_2 (E\otimes W))/((f^*_2 (E\otimes W))\otimes
{\mathcal O}_{M\times M}(-(k+1)\Delta_M)))
\end{equation}
$$
=\, f_{1*}(((f^*_2 E)\otimes (f^*_2 W))/((f^*_2 E)
\otimes (f^*_2W)\otimes {\mathcal O}_{M\times M}(-(k+1)\Delta_M)))
\, \longrightarrow\, M\, .
$$
Consider the pulled back connection $f^*_2 D_E$ on $f^*_2E$.
Since the sheaf of holomorphic two-forms on $M$ is the zero sheaf, the connection $D_E$ is
flat (the curvature vanishes identically). Hence the connection $f^*_2D_E$ is also flat. 
Therefore, on any analytic neighborhood $U$ of $\Delta_M\,\subset\, M\times M$ that
admits a deformation retraction to $\Delta_M$, the two vector bundles 
$(p^*_1E)\vert_U$ and $(p^*_2E)\vert_U$ are identified using parallel translation.
Invoking this isomorphism between $(p^*_1E)\vert_U$ and $(p^*_2E)\vert_U$, from \eqref{jt} we have
$$
J^k(E\otimes W)\, =\, 
f_{1*}(((f^*_1 E)\otimes (f^*_2 W))/((f^*_1 E)
\otimes (f^*_2W)\otimes {\mathcal O}_{M\times M} (-(k+1)\Delta_M)))\, .
$$
Hence the projection formula gives that
\begin{equation}\label{jt2}
J^k(E\otimes W)\, =\, E\otimes f_{1*}((f^*_2 W)/((f^*_2W)\otimes
{\mathcal O}_{M\times M} (-(k+1)\Delta_M)))\, =\, E\otimes J^k(W)\, .
\end{equation}

\begin{lemma}\label{lem2}
Let $M$ be a connected Riemann surface equipped with a projective structure $P$. Let $E$ be a 
holomorphic vector bundle on $M$ equipped with a holomorphic connection $D_E$. If $n\, <\, 0$ or 
$n\, \geq\, k$, then there is a canonical isomorphism
$$
J^k(E\otimes {\mathbf L}^n)\,=\, E\otimes{\mathbf L}^{n-k}\otimes {\mathcal V}_k\, .
$$
If $k\, > \, n\, \geq\, 0$, then there is a canonical isomorphism
$$
J^k(E\otimes{\mathbf L}^n)\,=\, E\otimes ({\mathcal V}_n\oplus ({\mathbf L}^{-(k+1)}\otimes
{\mathcal V}_{k-n-1}))\, .
$$
\end{lemma}

\begin{proof}
This follows from \eqref{jt2} and Corollary \ref{cor1}. We omit the details.
\end{proof}

\subsection{The symbol map}

As before, $M$ is a compact connected Riemann surface equipped with a projective structure $P$,
and $E$ is a holomorphic vector bundle over $M$ equipped with a holomorphic connection $D_E$.
Consider the projection
\begin{equation}\label{j3}
\phi_k \, :\, J^k(E\otimes{\mathbf L}^n)\, \longrightarrow\, J^{k-1}(E\otimes{\mathbf L}^n)
\end{equation}
in \eqref{e2}.
We will describe $\phi_k$ in terms of the isomorphisms in Lemma \ref{lem2}.

\subsubsection{Case where either $n < 0$ or $n \geq k$}

Assume that at least one of the following two conditions is valid:
\begin{enumerate}
\item $n \,<\, 0$, or

\item $n \,\geq\, k$.
\end{enumerate}
In view of Lemma \ref{lem2}, this implies that
\begin{equation}\label{j4}
J^k(E\otimes{\mathbf L}^n)\,=\, E\otimes{\mathbf L}^{n-k}\otimes {\mathcal V}_k\ \ \text{ and }\ \
J^{k-1}(E\otimes{\mathbf L}^n)\,=\, E\otimes{\mathbf L}^{n-k+1}\otimes {\mathcal V}_{k-1}
\end{equation}
when $k\, \geq\, 1$ (note that in \eqref{j3} we have $k\, \geq\, 1$). Using \eqref{j4}, the
projection $\phi_k$ in \eqref{j3} is a surjective homomorphism
\begin{equation}\label{j5}
\phi_k \, :\, E\otimes{\mathbf L}^{n-k}\otimes{\mathcal V}_k\, \longrightarrow\,
E\otimes{\mathbf L}^{n-k+1}\otimes {\mathcal V}_{k-1}\, .
\end{equation}

Recall that ${\mathcal V}_j\,=\, \text{Sym}^j({\mathcal V})$, and the line bundle $\mathbf L$
is a quotient of $\mathcal V$ (see \eqref{bl}). Therefore, we have a natural projection
\begin{equation}\label{vk}
\varpi_k\, :\, {\mathcal V}_k\, \longrightarrow\,{\mathbf L}\otimes {\mathcal V}_{k-1}\, .
\end{equation}
To describe $\varpi_k$ explicitly, let $q_{\mathbf L}\, :\, {\mathcal V}\, \longrightarrow\,
\mathbf L$ be the quotient map. Consider the homomorphism
$$
q_{\mathbf L}\otimes \text{Id}_{{\mathcal V}^{\otimes(k-1)}}\, :\, {\mathcal V}^{\otimes k}\,
\longrightarrow\,{\mathbf L}\otimes{\mathcal V}^{\otimes (k-1)}\, .
$$
Note that $q_{\mathbf L}\otimes \text{Id}_{{\mathcal V}^{\otimes(k-1)}}$ sends the subbundle
$\text{Sym}^k({\mathcal V})\, \subset\, {\mathcal V}^{\otimes k}$ to the subbundle
${\mathbf L}\otimes {\mathcal V}_{k-1}\, \subset\, {\mathbf L}\otimes{\mathcal V}^{\otimes (k-1)}$.
The homomorphism $\varpi_k$ in \eqref{vk} is this restriction of the homomorphism
$q_{\mathbf L}\otimes \text{Id}_{{\mathcal V}^{\otimes(k-1)}}$.

\begin{proposition}\label{prop5}
The homomorphism $\phi_k$ in \eqref{j5} coincides with ${\rm Id}\otimes\varpi_k$, where
${\rm Id}$ denotes the identity map of $E\otimes{\mathbf L}^{n-k}$, and $\varpi_k$ is
the homomorphism in \eqref{vk}.
\end{proposition}

\begin{proof}
This follows from Proposition \ref{prop2} and Proposition \ref{prop3}.
\end{proof}

Consider the inclusion homomorphism $\iota_{\mathcal L}\, :\, {\mathcal L}\,\hookrightarrow\,
\mathcal V$ (see \eqref{bl}). We have the subbundle
$$
\iota'\, :=\, \text{Id}_E\otimes{\iota}^{\otimes k}_{\mathcal L}\, :\, 
E\otimes{\mathcal L}^{\otimes k}
\,\hookrightarrow\, E\otimes\text{Sym}^k({\mathcal V})\, =\, E\otimes{\mathcal V}_k\, .
$$
Using the isomorphism ${\mathbf L}\,=\, {\mathcal L}^*$ in \eqref{l4}, the above homomorphism
$\iota'$ produces a homomorphism
\begin{equation}\label{sb}
\widetilde{\iota}\, :=\, \iota'\otimes\text{Id}_{{\mathbf L}^{n-k}}\, :\,
E\otimes {\mathbf L}^{n-2k}\,=\, E\otimes{\mathcal L}^{\otimes k}\otimes {\mathbf L}^{n-k}
\,\hookrightarrow\, E\otimes{\mathcal V}_k\otimes{\mathbf L}^{n-k}\, .
\end{equation}

The following is an immediate consequence of Proposition \ref{prop5}.

\begin{corollary}\label{cor2}
The kernel of the surjective homomorphism $\phi_k$ in \eqref{j5} coincides with the
subbundle $E\otimes{\mathbf L}^{n-2k}\,\hookrightarrow\,
E\otimes{\mathcal V}_k\otimes{\mathbf L}^{n-k}$ given by the image
of the homomorphism $\widetilde{\iota}$ in \eqref{sb}.
\end{corollary}

Let $F$ be any holomorphic vector bundle on $M$. Using \eqref{dod} and \eqref{j4}, we have
\begin{equation}\label{sb2}
\text{Diff}^k_M(E\otimes{\mathbf L}^n,\, F)\,=\, F\otimes (E\otimes
{\mathbf L}^{n-k}\otimes {\mathcal V}_k)^*
\,=\, F\otimes {\mathbf L}^{k-n}\otimes (E\otimes {\mathcal V}_k)^*\, .
\end{equation}
Using the homomorphism $\widetilde{\iota}$ in \eqref{sb} construct the homomorphism
\begin{equation}\label{sb1}
\text{Id}_{F}\otimes\widetilde{\iota}^*\, :\, 
F\otimes (E\otimes{\mathcal V}_k\otimes{\mathbf L}^{n-k})^* \, \longrightarrow\,
F\otimes (E\otimes {\mathbf L}^{n-2k})^*\,=\,
\text{Hom}(E,\, F)\otimes {\mathbf L}^{2k-n}\, .
\end{equation}
Combining \eqref{sb2} and \eqref{sb1}, we have a surjective homomorphism
$$
\sigma_k\, :\, \text{Diff}^k_M(E\otimes {\mathbf L}^n,\, F)\, \longrightarrow\,
\text{Hom}(E,\, F)\otimes {\mathbf L}^{2k-n}\, .
$$

The following is an immediate consequence of Corollary \ref{cor2}.

\begin{corollary}\label{cor3}
The symbol homomorphism
$$
{\rm Diff}^k_M(E\otimes {\mathbf L}^n,\, F)\, \longrightarrow\,
{\rm Hom}(E\otimes {\mathbf L}^n,\, F)\otimes (TM)^{\otimes k}
$$
in \eqref{e2p} coincides with the above homomorphism $\sigma_k$ after $TM$
is identified with ${\mathbf L}^{\otimes 2}$ as in \eqref{l4}.
\end{corollary}

\subsubsection{Case where $k > n \geq 0$}

We first assume that $k-1 \,>\, n \,\geq\, 0$. Therefore, from Lemma \ref{lem2} we have
$J^k(E\otimes{\mathbf L}^n)\,=\, E\otimes ({\mathcal V}_n\oplus ({\mathbf L}^{-(k+1)}\otimes
{\mathcal V}_{k-n-1}))$ and
$$
J^{k-1}(E\otimes{\mathbf L}^n)\,=\, E\otimes ({\mathcal V}_n\oplus
({\mathbf L}^{-k}\otimes {\mathcal V}_{k-n-2}))\, .
$$
Let
\begin{equation}\label{j6}
\phi_k\, :\, E\otimes ({\mathcal V}_n\oplus ({\mathbf L}^{-(k+1)}\otimes
{\mathcal V}_{k-n-1}))\,=\, J^k(E\otimes{\mathbf L}^n)
\end{equation}
$$
\longrightarrow\,
J^{k-1}(E\otimes{\mathbf L}^n)\,=\, E\otimes ({\mathcal V}_n\oplus ({\mathbf L}^{-k}\otimes
{\mathcal V}_{k-n-2}))
$$
be the projection in \eqref{e2}. Just as the homomorphism ${\rm Id}\otimes\varpi_k$
in Proposition \ref{prop5}, we have a homomorphism
$$
\widehat{\varpi}_k\, :\, E\otimes {\mathbf L}^{-(k+1)}\otimes
{\mathcal V}_{k-n-1}\, \longrightarrow\,E\otimes {\mathbf L}^{-k}\otimes
{\mathcal V}_{k-n-2}\, .
$$

\begin{lemma}\label{lem3}
The homomorphism $\phi_k$ in \eqref{j6} coincides with the homomorphism
$${\rm Id}_{E\otimes{\mathcal V}_n}\oplus \widehat{\varpi}_k
\, :\, E\otimes ({\mathcal V}_n\oplus ({\mathbf L}^{-(k+1)}\otimes
{\mathcal V}_{k-n-1}))\, \longrightarrow\,
E\otimes ({\mathcal V}_n\oplus ({\mathbf L}^{-k}\otimes
{\mathcal V}_{k-n-2}))\, .
$$
\end{lemma}

\begin{proof}
This follows from Proposition \ref{prop4}.
\end{proof}

Note that using \eqref{l4},
$$\text{kernel}(\widehat{\varpi}_k)\,=\,E\otimes {\mathbf L}^{-(k+1)}\otimes
{\mathcal L}^{\otimes (k-n-1)}\,=\, E\otimes K^{\otimes k}_M{\mathbf L}^{n}\, .$$

The following is an immediate consequence of Lemma \ref{lem3}.

\begin{corollary}\label{cor5}
Let $F$ be any holomorphic vector bundle on $M$. The symbol homomorphism
$$
{\rm Diff}^k_M(E\otimes {\mathbf L}^n,\, F)\, \longrightarrow\,
{\rm Hom}(E\otimes {\mathbf L}^n,\, F)\otimes (TM)^{\otimes k}
$$
coincides with ${\rm Id}_F\otimes \iota^*$, where $\iota\, :\,
{\rm kernel}(\widehat{\varpi}_k)\,\hookrightarrow\, E\otimes {\mathbf L}^{-(k+1)}\otimes
{\mathcal V}_{k-n-1}$ is the inclusion map.
\end{corollary}

We now assume that $k-1 \,=\, n \,\geq\, 0$.

Under this assumption, from Lemma \ref{lem2} we have
$$
J^k(E\otimes{\mathbf L}^n)\,=\, E\otimes ({\mathcal V}_n\oplus {\mathbf L}^{-(k+1)}) \ \ 
\text{ and }\ \ J^{k-1}(E\otimes{\mathbf L}^n)\,=\, 
E\otimes{\mathcal V}_n\, .
$$
Let
\begin{equation}\label{j6a}
\phi_k\, :\, E\otimes ({\mathcal V}_n\oplus {\mathbf L}^{-(k+1)})\,=\,
J^k(E\otimes{\mathbf L}^n)\, \longrightarrow\,
J^{k-1}(E\otimes{\mathbf L}^n)\,=\, E\otimes{\mathcal V}_n 
\end{equation}
be the projection in \eqref{e2}.

\begin{lemma}\label{lem4}
The homomorphism $\phi_k$ in \eqref{j6a} coincides with the projection
$$
E\otimes ({\mathcal V}_n\oplus {\mathbf L}^{-(k+1)})\, \longrightarrow\,
E\otimes{\mathcal V}_n 
$$
to the first factor.
\end{lemma}

\begin{proof}
This follows from Proposition \ref{prop4}.
\end{proof}

The following is an immediate consequence of Lemma \ref{lem4}.

\begin{corollary}\label{cor6}
Let $F$ be any holomorphic vector bundle on $M$. The symbol homomorphism
$$
{\rm Diff}^k_M(E\otimes {\mathbf L}^n,\, F)\, \longrightarrow\,
{\rm Hom}(E\otimes {\mathbf L}^n,\, F)\otimes (TM)^{\otimes k}
$$
coincides with ${\rm Id}_F\otimes p'$, where
$$p'\, :\, E^*\otimes ({\mathcal V}_n\oplus {\mathbf L}^{-(k+1)})^*\,\longrightarrow\,
E^*\otimes ({\mathbf L}^{-(k+1)})^*
$$
is the natural projection.
\end{corollary}

\section{Lifting of symbol}

As before, $M$ is a connected Riemann surface equipped with a projective
structure $P$, and $E$ is a holomorphic vector bundle on $M$ equipped
with a holomorphic connection $D_E$. The connection $D_E$ on $E$ induces
a holomorphic connection on $\text{End}(E)$. This induced connection on
$\text{End}(E)$ will be denoted by ${\mathcal D}_E$.

Take integers $k$, $n$ and $l$ such that at least one of the following two conditions is valid:
\begin{enumerate}
\item $n \,<\, 0$, or

\item $n,\, l \,\geq\, k$.
\end{enumerate}

Let
\begin{equation}\label{sy}
\theta_0\, \in\, H^0(M,\, \text{End}(E)\otimes{\mathbf L}^l)
\end{equation}
be a holomorphic section. Note that $\theta_0$ defines a section
$$
\theta\, \in\, H^0(M,\, J^k(\text{End}(E)\otimes{\mathbf L}^l))\, .
$$
Recall that $D_E$ induces a holomorphic connection ${\mathcal D}_E$ on $\text{End}(E)$. In view
of this connection, Lemma \ref{lem2} says that $J^k(\text{End}(E)\otimes {\mathbf L}^l)\,=\,
\text{End}(E)\otimes{\mathbf L}^{l-k}\otimes {\mathcal V}_k$. So, we have
\begin{equation}\label{sy2}
\theta\, \in\, H^0(M,\, \text{End}(E)\otimes{\mathbf L}^{l-k}\otimes {\mathcal V}_k)\, .
\end{equation}

Recall from Lemma \ref{lem2} that $J^k(E\otimes{\mathbf L}^n)\,=\, E\otimes{\mathbf L}^{n-k}\otimes
{\mathcal V}_k$. Let
$$
{\mathcal T}^0_\theta\, :\, J^k(E\otimes{\mathbf L}^n)\,=\, E\otimes{\mathbf L}^{n-k}\otimes 
{\mathcal V}_k\, \longrightarrow\, E\otimes{\mathbf L}^{l+n-2k}\otimes{\mathcal 
V}_k\otimes{\mathcal V}_k
$$
be the homomorphism defined by $a\otimes b\, \longmapsto\, \sum_i(a'_i(a)\otimes b\otimes 
b'_i)$, where $a$ and $b$ are local sections of $E$ and ${\mathbf L}^{n-k}\otimes {\mathcal 
V}_k$ respectively, while $\theta$ in \eqref{sy2} is locally expressed as $\sum_i a'_i\otimes 
b'_i$ with $a'_i$ and $b'_i$ being local sections of $\text{End}(E)$ and ${\mathbf 
L}^{l-k}\otimes {\mathcal V}_k$ respectively. Let
$$
p_0\, :\, {\mathcal V}_k\otimes{\mathcal V}_k\, \longrightarrow, {\mathcal V}_0\,=\,
{\mathcal O}_M
$$
be the projection constructed using the decomposition in Proposition \ref{prop1}. Now define the homomorphism
\begin{equation}\label{sy3}
{\mathcal T}_\theta\,:=\, (\text{Id}\otimes p_0)\circ {\mathcal T}^0_\theta\, :\,
J^k(E\otimes{\mathbf L}^n)\,=\, E\otimes{\mathbf L}^{n-k}\otimes
{\mathcal V}_k\, \longrightarrow\, E\otimes{\mathbf L}^{l+n-2k}\, ,
\end{equation}
where $\text{Id}$ denotes the identity map of $E\otimes{\mathbf L}^{l+n-2k}$. Therefore, we have
\begin{equation}\label{sy4}
{\mathcal T}_\theta\, \in\,
H^0(M,\, \text{Diff}^k(E\otimes{\mathbf L}^n,\, E\otimes{\mathbf L}^{l+n-2k}))
\end{equation}
(see \eqref{dod}). Let
$$
\widehat{\sigma}\, :\, \text{Diff}^k(E\otimes{\mathbf L}^n,\, E\otimes{\mathbf L}^{l+n-2k})\, \longrightarrow\,
\text{End}(E)\otimes {\mathbf L}^{l-2k}\otimes(TM)^{\otimes k}\,=\, \text{End}(E)\otimes {\mathbf L}^l
$$
be the symbol homomorphism in \eqref{e2p}; in \eqref{l4} it was shown that $TM\,=\, {\mathbf L}^{\otimes 2}$.

\begin{theorem}\label{thm2}
The symbol $\widehat{\sigma}({\mathcal T}_\theta)$ of the differential operator ${\mathcal T}_\theta$
in \eqref{sy4} is the section $\theta_0$ in \eqref{sy}.
\end{theorem}

\begin{proof}
This theorem can be derived from Corollary \ref{cor3} that describes the symbol. The
following is needed to derive the theorem from Corollary \ref{cor3}.

Consider the decomposition of the $\text{SL}({\mathbb V})$--module
$$
\text{Sym}^{m}({\mathbb V})\otimes \text{Sym}^{m}({\mathbb V})\,=\,
\bigoplus_{i=0}^m \text{Sym}^{2(m-i)}({\mathbb V})
$$
in \eqref{td}. Let
$$
f\, :\, \text{Sym}^{m}({\mathbb V})\otimes \text{Sym}^{m}({\mathbb V})\,\longrightarrow\,
\text{Sym}^{0}({\mathbb V})\,=\, \mathbb C
$$
be the projection constructed using this decomposition (by setting $i\,=\, m$
in the decomposition). Then for any line
$L_0\, \subset\, \mathbb V$, we have
$$
f((L^{\otimes m}_0)\otimes (L_0\otimes\text{Sym}^{m-1}({\mathbb V})))\,=\, 0\, ,
$$
so $f$ produces a homomorphism
$$
\widehat{f}\, :\, (L^{\otimes m}_0)\otimes ((\text{Sym}^{m}({\mathbb V}))/
(L_0\otimes\text{Sym}^{m-1}({\mathbb V})))\,=\, (L^{\otimes m}_0)\otimes 
({\mathbb V}/L_0)^{\otimes m}\,\longrightarrow\, {\mathbb C}\, .
$$
This homomorphism $\widehat{f}$ sends $v^{\otimes m}\otimes u^{\otimes m}$,
where $v\, \in\, L_0$ and $u\,\in\, {\mathbb V}/L_0$, to $\lambda^m\, \in\, \mathbb C$
that satisfies the equation
$$
v\wedge u\,=\, \lambda\cdot \omega\, ,
$$
where $\omega$ is the fixed element in \eqref{e5}; note that $v\wedge u$ is a well-defined element
of $\bigwedge^2 \mathbb V$. The theorem follows from this observation and Corollary \ref{cor3}.
\end{proof}

\begin{remark}
It may be mentioned that symbols do not always lift to a holomorphic differential
operator. For such an example, set the symbol to be the constant function $1$. It does
not lift to a first order holomorphic differential operator from ${\mathcal L}^n$
to ${\mathcal L}^{n}\otimes K_M$ if $g\, \geq\, 2$ and $n\, \not=\, 0$. Indeed, a
first order holomorphic differential operator from ${\mathcal L}^n$
to ${\mathcal L}^{n}\otimes K_M$ with symbol $1$ is a holomorphic connection on
${\mathcal L}^n$. But ${\mathcal L}^n$ does not admit a holomorphic connection because
its degree is nonzero.
\end{remark}

\section*{Acknowledgements}

This work has been supported by the French government through the UCAJEDI Investments in the 
Future project managed by the National Research Agency (ANR) with the reference number 
ANR2152IDEX201. The first author is partially supported by a J. C. Bose Fellowship.


\end{document}